\documentclass{article}
\usepackage[utf8]{inputenc}
\usepackage{amsmath,amssymb,amsfonts,amsthm,enumitem,xy,pstricks,pst-node,graphicx,comment}
\xyoption{all}

\title{Quantitative Smoothing of Polyhedral Manifolds}
\author{Spencer Cattalani\thanks{Partially supported by NSF grant DMS 2246485 and the Simons Foundation}}
\date{\today}

\newtheorem{thm}{Theorem}
\newtheorem{prp}[thm]{Proposition}
\newtheorem{lmm}[thm]{Lemma}   
\newtheorem{crl}[thm]{Corollary} 
\newtheorem{eg}[thm]{Example}

\theoremstyle{definition}
\newtheorem{dfn}[thm]{Definition}

\theoremstyle{remark}
\newtheorem{rmk}[thm]{Remark}

\theoremstyle{remark}

\def\BE#1{\begin{equation}\label{#1}}
\def\EE{\end{equation}}
\def\e_ref#1{(\ref{#1})}

\def\lra{\longrightarrow}

\def\R{\mathbb R}
\def\Z{\mathbb Z}

\begin{document}

\maketitle

\begin{abstract}
We use a recent result of C.~Lange to obtain a converse to a theorem of B.~Bowditch in dimension at most~$4$. In particular, we show that, for $n \leq 4$, a polyhedral $n$-manifold $X$ with bounded geometry is $K$-bi-Lipschitz homeomorphic to a Riemannian manifold $M$. We bound the constant $K$, the curvature, and the injectivity radius of $M$ by the bounds on the geometry of $X$.
\end{abstract}

\section{Introduction}
A major theme in metric geometry is uniformizing or approximating metric spaces; see, e.g. \cite{Cre,Nta}. Two common classes of such approximations are simplicial complexes and Riemannian manifolds. Different techniques are available in each setting and it is helpful to be able to switch between them. For applications, it is important to retain bounds on the geometry while doing so. In one direction, Bowditch \cite{Bow} (cf.~\cite{Boi}) showed that a Riemannian manifold with curvature and injectivity radius bounds can be triangulated by a simplicial complex with bounded geometry. He, furthermore, showed that the triangulation is $K$-bi-Lipschitz, where $K$ depends only on the geometry of the manifold. Using a method from geometric topology, we prove a partial converse to this result.\\

We recall that a \textit{polyhedral manifold} is a simplicial complex in which the link of every simplex is a sphere. A polyhedral manifold with a path metric such that each simplex is isometric to a Euclidean simplex is called a \textit{polyhedral metric manifold}. It is clear that the metric is determined by the lengths of the 1-simplices. A simplex with all side lengths equal is called \textit{standard}. The following definitions come essentially from \cite{Nta1} and \cite{Bow}, respectively.

\begin{dfn}
Let $(X,d)$ be a polyhedral metric manifold and $M \in \Z^+$. We will call $X$
\begin{itemize}
\item \textit{M-quasiconformal} if for each point $x \in X$ there exists $r > 0$ such that $x$ is contained in at most $M$ simplices and each such simplex is linearly $M$-bi-Lipschitz homeomorphic to the standard simplex with side length $r$.
\item \textit{M-Lipschitz} if each point of $X$ is contained in at most $M$ simplices and each simplex is linearly $M$-bi-Lipschitz homeomorphic to the standard simplex with side length $1$.
\end{itemize}
\end{dfn}

The following theorem is the main result of this paper. It plays a crucial role in the main construction of \cite{Nta1} and should be helpful in related constructions. We emphasize that the constant $K$ below depends only on the constant $M$ and the result is therefore quantitative.

\begin{thm}\label{main_thm}
Let $n \leq 4$ and $M \in \Z^+$. Then, there exists $K \geq 1$ such that:
\begin{itemize}
\item If $(X,d)$ is an $M$-quasiconformal polyhedral $n$-manifold, then it is $K$-bi-Lipschitz homeomorphic to a Riemannian manifold.
\item If $(X,d)$ is an $M$-Lipschitz polyhedral $n$-manifold, then it is $K$-bi-Lipschitz homeomorphic to a Riemannian manifold with sectional curvatures between $-K$ and $K$ and injectivity radius greater than $K^{-1}$. Furthermore, all the covariant derivatives of the curvature are bounded.
\end{itemize}
Moreover, in both cases the construction is equivariant; every simplicial isometry of $(X,d)$ induces (via the homeomorphism) a smooth isometry of the Riemannian manifold.
\end{thm}

We note that the boundedness of the covariant derivatives also follows from~\cite{Bem}. Theorem~\ref{main_thm} is false if the assumption on dimension is removed completely, because there are polyhedral manifolds which are not homeomorphic to any smooth manifold \cite{Kui}. One could ask if the assumption on dimension could be replaced by the assumption of being homeomorphic (or piecewise-diffeomorphic, or locally Lipschitz homeomorphic) to a smooth manifold. These questions are quite difficult, but may be related to the fact that the optimal constant of a bi-Lipschitz homeomorphism can detect exotic smooth structures \cite{Shi,Wel}. In any case, our proof of Theorem~\ref{main_thm} is essentially low-dimensional, and a key lemma in the argument (Lemma~\ref{canonical_smoothing_lmm}) is known to fail in higher dimension \cite{Lan}.\\

In dimension 2, we are able to make the curvature bound appearing in Theorem~\ref{main_thm} effective (cf. \cite{Bre,Cho}). We are also able to weaken the assumptions. In particular, we only need rather weak geometric assumptions and no combinatorial assumptions. This is the content of our second result. We recall the notion of cone angle in Section~\ref{dim_2_sec}.

\begin{thm}\label{dim_2_thm}
Let $(X,d)$ be a polyhedral metric surface and $K \geq 1$. If the cone angle at every vertex is between $2\pi K^{-1}$ and $2\pi K$, then $(X,d)$ is $K$-bi-Lipschitz to a Riemannian manifold $\widetilde{X}$. If $l$ is the infimal distance between vertices, then $\widetilde{X}$ can be chosen to have Gaussian curvature bounded in absolute value by
$2^9K(K\!-\!1)/l^2$.
\end{thm}

The constant $2^9$ is not sharp. It is clear that a bound on curvature must depend on a length squared, for dimensional consistency. The reason for the quadratic dependence on $K$ is because the integral of the curvature over the area should give the angle defect and both the area and angle defect are linearly bounded by $K$. The fact that one cannot bound the Lipschitz constant without a bound on the cone angle is the content of the following proposition.

\begin{prp}\label{obstruction_prp}
Let $(X,d)$ be a polyhedral metric surface. If it is $K$-bi-Lipschitz to a Riemannian manifold, then the cone angle at every vertex is between $2 \pi K^{-2}$ and $2 \pi K^2$.
\end{prp}

\begin{rmk}
The proof of Proposition~\ref{obstruction_prp} generalizes to higher dimension to prove the following statement.
Let $X$ be a polyhedral metric manifold which is $K$-bi-Lipschitz to a Riemannian $(n+1)$-manifold. For $\varepsilon > 0$ sufficiently small, the $n$-area of an $\varepsilon$-sphere around a point 
is between $\varepsilon^n\omega_{n}K^{-n-1}$ and $\varepsilon^n\omega_{n}K^{n+1}$, where $\omega_n$ is the $n$-volume of the unit $n$-sphere in Euclidean space.
\end{rmk}

\textbf{Acknowledgements:} The author thanks Dimitrios Ntalampekos for bringing this topic to his attention and for his help improving the paper, Aleksey Zinger for many helpful comments and discussions, Christian Lange for clarifiying remarks on his results in \cite{Lan}, and the anonymous referee for helpful comments which strengthened Theorem~\ref{main_thm}.

\section{Effective Bounds in Dimension 2}\label{dim_2_sec}

For $\alpha > 0$, let $g_\alpha$ denote the metric on $\R^2$ defined in polar coordinates by
$$g_\alpha = \Big(\frac{\alpha r}{2\pi}\Big)^2 d\theta^2 + dr^2.$$
If $\alpha = 2\pi$, this is the Euclidean metric.

\begin{dfn}\label{cone_dfn}
A \textit{cone} is a metric space isometric to a disk around $0$ in $(\R^2,g_\alpha)$. The point $0$ is called the \textit{tip} and the number $\alpha$ is called the \textit{cone angle}.
\end{dfn}

\begin{lmm}\label{angle_meaning_lmm}
The cone angle is the length of the shortest curve enclosing a disk of radius $r$ around the tip, divided by $r$.
\end{lmm}

\begin{proof}
It is easily calculated that the length of the circle of radius $r$ divided by $r$ is the cone angle. The circle is the shortest curve encircling the disk, because the radial projection onto the circle is $1$-Lipschitz.
\end{proof}

\begin{crl}\label{cone_angle_crl}
Let $C_1$ and $C_2$ be cones with tips $p_1$ and $p_2$, respectively. If there is a $K$-bi-Lipschitz homeomorphism $f: C_1 \lra C_2$ such that $f(p_1) = p_2$, then
$$K^{-2} \textnormal{angle}(C_1) \leq \textnormal{angle}(C_2) \leq K^2\textnormal{angle}(C_1).$$
\end{crl}

\begin{proof}
Let $\gamma_r$ be the circle of radius $r$ around $p_1$. It has length $\textnormal{angle}(C_1)r$. As $f$ is $K$-Lipschitz, $f(\gamma_r)$ has length at most $K\textnormal{angle}(C_1)r$ and encloses the disk of radius $K^{-1}r$. By Lemma~\ref{angle_meaning_lmm}, $\textnormal{angle}(C_2) \leq K^2\textnormal{angle}(C_1)$. The other inequality follows by considering $f^{-1}$.
\end{proof}

\begin{proof}[Proof of Proposition~\ref{obstruction_prp}]
For every $\varepsilon > 0$ and point $x$ in a Riemannian surface, $x$ has a neighborhood $(1+ \varepsilon)$-bi-Lipschitz to a cone with angle $2\pi$. The result then follows from Corollary~\ref{cone_angle_crl}.
\end{proof}

\begin{proof}[Proof of Theorem~\ref{dim_2_thm}]
For each point $x \in X$, let $\rho(x)$ be half the distance to the nearest vertex (distinct from~$x$). The disk of radius $\rho(x)$ around $x$ is isometric to a cone. If $x$ is not a vertex, this cone is a disk in the plane. The isometries between disks around points in $X$ and disks in the plane define smooth (actually affine) coordinates on $X$, away from the vertices. The charts coming from the cones at vertices are compatible with the smooth atlas defined on the complement of the vertices. Therefore, we have defined a smooth structure on $X$.\\

We now smooth out the metric. Let $f$ be a smooth positive function of $r$ that vanishes to infinite order as $r$ approaches zero and equals $1$ for $r$ greater than $1$, such that $f$, $|f'|$, and $|f''|$ are all bounded above by $2^4$. Let $v$ be a vertex with cone angle $\alpha$ and $f_v(r) = f(\frac{r}{\rho(v)})$. Then the metric
$$\widetilde{g_\alpha} := \Bigg(f_v(r)\Big(\frac{\alpha}{2\pi} - 1\Big) + 1\Bigg)^2r^2d\theta^2 + dr^2$$
is smooth on a disk of radius $\rho(v)$ around $v$. It converges to the polyhedral metric as $r$ tends to~$\rho(v)$. As $\widetilde{g_\alpha}$ is a pointwise convex combination of $g_\alpha$ and the Euclidean metric, it is $K$-bi-Lipschitz to $g_\alpha$. Replacing $g_\alpha$ with $\widetilde{g_\alpha}$ for each vertex yields a Riemannian metric on $X$ which is $K$-bi-Lipschitz to the polyhedral metric.\\

We now bound the curvature of $\widetilde{g_\alpha}$ assuming $l = 2$ (i.e. $\rho(v) \geq 1$ for all vertices $v$). This can be achieved by scaling the metric, so there is no loss of generality. Plugging this into the formula for Gaussian curvature in orthogonal coordinates \cite[Section 3.4]{Opr}  yields
$$\textnormal{Curvature} = -\frac{1}{2\phi(r)}\Bigg(\frac{((\phi(r))^2)'}{\phi(r)}\Bigg)',$$
where $\phi(r) = \Big(f_v(r)\big(\frac{\alpha}{2\pi} - 1\big) + 1\Big)r$. A calculation reduces this to
$$\frac{(2\pi - \alpha)(2f_v'(r)+rf_v''(r))}{2\pi\phi(r)}.$$
The assumptions on $\rho(v)$ and the derivatives of $f$ allow us (via the chain rule) to bound the absolute value of the numerator by $96\pi(K-1)$. As $2\pi\phi(r)$ is a convex combination of $\alpha$ and $2\pi$, it is bounded below by the minimum of the two. Lastly, we must add a factor of $4/l^2$ to account for the normalization $l = 2$, which yields the claimed inequality.
\end{proof}

\section{Canonical Smoothing in Low Dimension}\label{low_dim_sec}
The proof of Theorem~\ref{main_thm} follows the same main argument as Theorem~\ref{dim_2_thm}. Broadly, one takes a standard model around each singular point, smooths it out, and patches the resulting metrics together. However, there are some additional difficulties. The first is that it is not so easy to construct local models for the metric. For surfaces, this is quite easy and we are able to parameterize them as a continuous family depending only on the cone angle. In dimensions $3$ and $4$, we instead rely on Lemma~\ref{canonical_smoothing_lmm} and must therefore account for different combinatorial types. The second complication is that the set of points for which the metric is singular can fail to be discrete in dimensions higher than $2$. This means that we must patch the smooth local metrics together using a partition of unity. Lemma~\ref{locality_lmm} ensures that this step retains control of the Lipschitz constant regardless of the partition of unity. In order to control the curvature and injectivity radius, we choose a standard partition. This choice is akin to the choice of $f$ in the proof of Theorem~\ref{dim_2_thm}.\\

Lemma~\ref{canonical_smoothing_lmm} below states that there is a canonical way to smooth triangulated manifolds of low dimension. In particular, the smoothing depends only on the local combinatorial structure. It was first stated for dimensions $2$ and $3$ in \cite[pp.207,208]{Thu}. The lemma follows from the proof of the main theorem in~\cite{Lan}.

\begin{lmm}[\cite{Lan}]\label{canonical_smoothing_lmm}
Let $n \leq 4$. There is a canonical differentiable structure on every triangulated $n$-manifold such that every combinatorial isomorphism $f: U \lra V$ between open star neighborhoods in triangulated $n$-manifolds is a diffeomorphism. Furthermore, if $X$ is a triangulated $n$-manifold, there is a subdivision of $X$ in which all the simplices are smoothly embedded with respect to the differentiable structure.
\end{lmm}

The next lemma shows that locally defined Riemannian metrics can be assembled into a globally defined metric while retaining bi-Lipschitz control. Throughout, when restricting to a subset, we consider the intrinsic length metric of the subset. We also note that two length metrics on a space $X$ are $K$-bi-Lipschitz if they assign the same length to each path, up to a factor of $K$. By subdividing, it suffices to check this for short segments of paths.

\begin{lmm}\label{locality_lmm}
Let $X$ be a smooth manifold and $d$ be a length metric on $X$. Suppose $\{U_\alpha\}$ is an open cover of $X$, $\{\phi_\alpha\}$ is a partition of unity subordinate to $\{U_\alpha\}$, and $\{g_\alpha\}$ is a collection of Riemannian metrics such that each $g_\alpha$ is $K$-bi-Lipschitz to $d$ on $U_\alpha$. Then,
$$g := \sum \phi_\alpha g_\alpha$$
is $K$-bi-Lipschitz to $d$ on $X$.
\end{lmm}

\begin{proof}
At any point $x \in X$, $g|_x$ is a convex sum of finitely many $g_\alpha|_x$. Therefore, $g|_x(v,v)$ is in the convex hull of $g_\alpha|_x(v,v) \in \R$ for all $v \in T_x X$. As the square root function is monotonic, there exist $\alpha$ and $\alpha'$ such that for all $v \in T_x X$,
\begin{equation}\label{convex_sum_eq}
\sqrt{g_\alpha|_x(v,v)} \leq \sqrt{g|_x(v,v)} \leq \sqrt{g_{\alpha'}|_x(v,v)}.
\end{equation}
Integrating (\ref{convex_sum_eq}) over very short segments of a smooth curve and using the fact that each $g_\alpha$ is $K$-bi-Lipschitz to $d$ shows that $g$ is $K$-bi-Lipschitz to $d$.
\end{proof}

The proof of Theorem~\ref{main_thm} is below. The idea is as follows. On an $M$-quasiconformal polyhedral manifold, there are only finitely many combinatorial types locally. The metric is locally determined up to scale and $M$-bi-Lipschitz homeomorphism by the combinatorial structure. Using the canonical smoothing of Lemma~\ref{canonical_smoothing_lmm}, we can smooth each of these local models in a globally coherent way, and use Lemma~\ref{locality_lmm} to glue the local smoothings together.

\begin{proof}[Proof of Theorem~\ref{main_thm}]
Let $\mathcal{T}$ be the collection of polyhedral $n$-balls $S$ with at most $M$ simplices, all of which contain some vertex $v_S$ in common. Equivalently, $\mathcal{T}$ can be thought of as the set of possible stars around vertices of an $M$-quasiconformal manifold. Fix a Riemannian metric $g_S$ and a compactly supported positive function $\phi_S$ on each $S \in \mathcal{T}$ which are smooth with respect to the canonical smooth structure given by Lemma~\ref{canonical_smoothing_lmm} and such that $\phi_S = 1$ on the star of $v_S$ in the first barycentric subdivision of~$S$. By averaging, we may assume both $g_S$ and $\phi_S$ are invariant under the (finite) simplicial symmetry group of S. Let $d_{S,r}$ be the polyhedral metric on $S$ in which all $1$-simplices have length $r$. The simplices of some subdivision of $S$ are smoothly embedded by Lemma~\ref{canonical_smoothing_lmm}. Therefore, for each $S$, there exists $K_S \geq 1$ such that $g_S$ is $K_S$-bi-Lipschitz to $d_{S,1}$. Therefore, $d_{S,r}$ is $K_S$-bi-Lipschitz to $r^2 g_S$. \\

Let $(X,d)$ be an $M$-quasiconformal polyhedral $n$-manifold, $v$ be a vertex of $X$, and $\textnormal{St}(v)$ be the star of~$v$. By $M$-quasiconformality, $\textnormal{St}(v) \in \mathcal{T}$ and there is a minimal $r > 0$ such that $d$ is $M$-bi-Lipschitz to $d_{\textnormal{St}(v),r}$. Therefore, the Riemannian metric $r^2 g_{\textnormal{St}(v)}$ on $\textnormal{St}(v)$ is $MK_S$-bi-Lipschitz to $d$.\\

Thus, $\{\textnormal{St}(v)\}$ is an open cover of $X$, and each $\textnormal{St}(v)$ supports a metric which is $M\max_{S \in T}(K_S)$-bi-Lipschitz to $d$. Therefore, by Lemma~\ref{locality_lmm}, there is a Riemannian metric on $X$ that is $M\max_{S \in T}(K_S)$-bi-Lipschitz to $d$. The number $M\max_{S \in T}(K_S)$ depends only on $M$, as claimed.\\

Now, assume $X$ is $M$-Lipschitz. Then, $r$, as above, can always be chosen to be $1$. By definition, the sum of the functions $\phi_{\textnormal{St}(v)}$ is everywhere positive. If it is $1$ everywhere, the functions $\phi_{\textnormal{St}(v)}$ form a partition of unity. This can easily be obtained by normalizing. On each star, this normalization depends only on the stars of neighboring points. There are only finitely many such combinations, and therefore only finitely many possible normalizations. Therefore, the restriction to a star of the metric $g$ that results from gluing $g_{\textnormal{St}(v)}$ along the partition of unity is isometric to one of only finitely many metrics.
There exists $\varepsilon > 0$ such that an $\varepsilon$-ball around each point of $X$ is contained in some star. Therefore, the restriction of $g$ to an $\varepsilon$-ball around each point varies in a compact family and thus satisfies bounds on any locally-defined quantity, such as curvature and injectivity radius. The construction of the metric $g$ depends only on the functions $\phi_{\textnormal{St}(v)}$ and metrics $g_{\textnormal{St}(v)}$ (or $r^2 g_{\textnormal{St}(v)}$ in the $M$-quasiconformal case), which are invariant under the simplicial isometry group of $(X,d)$; it is therefore manifestly equivariant.
\end{proof}

\vspace{.1in}

{\it Department of Mathematics, Stony Brook University, Stony Brook, NY 11794\\
spencer.cattalani@stonybrook.edu}


\begin{thebibliography}{9}

\bibitem{Bem} J.~Bemelmans, Min-Oo, and E.A.~Ruh,
{\it Smoothing Riemannian manifolds},
Math.~Z.~188 (1984), no.~1, 69--74.


\bibitem{Boi} J.-D.~Boissonnat, R.~Dyer, and A.~Ghosh,
{\it Delaunay triangulation of manifolds},
Found.~Comput.~Math.~18 (2018), no.~2, 399--431

\bibitem{Bow} B.~Bowditch,
{\it Bilipschitz triangulations of {R}iemannian manifolds},
https://homepages.warwick.ac.uk/$\sim$masgak/preprints.html (2020)

\bibitem{Bre} U.~Brehm and W.~K\"uhnel,
{\it Smooth approximation of polyhedral surfaces regarding curvatures},
Geom.~Dedicata~12 (1982), no.~4, 435--461

\bibitem{Cho} S.~Chowdhury, H.~Hu, M.~Romney, and A.~Tsou,
{\it On CAT($\kappa$) surfaces},
arXiv: 2309.13533 (2023)

\bibitem{Cre} P.~Creutz and M.~Romney,
{\it Triangulating metric surfaces},
Proc.~Lond.~Math.~Soc.~125 (2022), no.~6, 1426--1451

\bibitem{Kui} N.~Kuiper,
{\it Algebraic equations for nonsmoothable {$8$}-manifolds},
Inst.~Hautes \'Etudes Sci.~Publ.~Math.~33 (1967), 139--155

\bibitem{Lan} C.~Lange,
{\it Equivariant smoothing of piecewise linear manifolds},
Math.~Proc.~Cambridge Philos.~Soc.~164 (2018), no.~2, 369--380

\bibitem{Nta1} D.~Ntalampekos,
{\it Characterization of quasispheres via smooth approximation}
arXiv: 2502.10006 (2025)
\bibitem{Nta} D.~Ntalampekos and M.~Romney,
{\it Polyhedral approximation of metric surfaces and applications to uniformization},
Duke Math.~J.~172 (2023), no.~9, 1673--1734

\bibitem{Opr} J.~Oprea,
{\it Differential Geometry and its Applications},
Prentice-Hall (1997)

\bibitem{Shi} Y.~Shikata,
{\it On the smoothing problem and the size of a topological manifold},
Osaka Math.~J.~3 (1966), 293--301

\bibitem{Thu} W.~Thurston,
{\it Three-Dimensional Geometry and Topology, {V}ol. 1},
Princeton University Press, 1997

\bibitem{Wel} G.~Weller,
{\it Equivalent sizes of {L}ipschitz manifolds and the smoothing problem},
Osaka Math.~J.~10 (1973), 507--510

\end{thebibliography}
\end{document}